\documentclass[10pt]{conm-p-l}
\usepackage{amsfonts,amsmath,mathtools,physics,dsfont}
\usepackage{amsthm,amssymb,graphicx,textcomp,stmaryrd}
\usepackage{siunitx}
\usepackage{enumerate}
\usepackage{mathrsfs}
\usepackage{tikz-cd}
\usepackage[all]{xy}
\usepackage{amsmath}
\usepackage{hyperref} 

\usepackage{url}
\usepackage{comment}
\usepackage{ulem}

\newcommand{\C}{\mathds{C}}
\newcommand{\R}{\mathds{R}}
\newcommand{\Q}{\mathds{Q}}

\newcommand{\PP}{\mathds{P}}

    \DeclareFontFamily{U}{wncy}{}
    \DeclareFontShape{U}{wncy}{m}{n}{<->wncyr10}{}
    \DeclareSymbolFont{mcy}{U}{wncy}{m}{n}
    \DeclareMathSymbol{\Sha}{\mathord}{mcy}{"58}

\newcommand{\ord}{\textnormal{ord}}

\definecolor{myorange}{rgb}{0.9, 0.55, 0.3}
\definecolor{mygreen}{rgb}{0.35, 0.71, 0.0}
\definecolor{mybrown}{rgb}{0.63, 0.32, 0.18}

\newtheorem{theorem}{Theorem}[section]

\newtheorem{lemma}[theorem]{Lemma}

\newtheorem{conjecture}[theorem]{Conjecture}
\theoremstyle{definition}
\newtheorem{definition}[theorem]{Definition}
\theoremstyle{definition}
\theoremstyle{definition}
\theoremstyle{definition}
\newtheorem{remark}[theorem]{Remark}

\title[Elliptic Lehmer Investigations]{Experimental investigations on Lehmer's conjecture for elliptic curves}

\author[]{Sven Cats}
\address{Centre for Mathematical Sciences, University of Cambridge, Cambridge, UK}
\email{sc2173@cam.ac.uk}

\author[]{John Michael Clark}
\address{Department of Mathematics, University of Texas at Austin, Austin, Texas, USA}
\email{john.m.clark@utexas.edu}

\author[]{Charlotte Dombrowsky}
\address{Mathematical Institute, Leiden University, Leiden, The Netherlands}
\email{c.k.l.dombrowsky@math.leidenuniv.nl}

\author[]{Mar Curc\'o Iranzo}
\address{Mathematical Institute, Utrecht University, Utrecht, The Netherlands}
\email{m.curcoiranzo@uu.nl}

\author[]{Krystal Maughan}
\address{Department of Computer Science, University of Vermont, Burlington, Vermont, USA}
\email{Krystal.Maughan@uvm.edu}

\author[]{Eli Orvis}
\address{Department of Mathematics, University of Colorado at Boulder, Boulder, Colorado, USA}
\email{eli.orvis@colorado.edu}


\begin{document}

\begin{abstract}
In this short note, we give a method for computing a non-torsion point of smallest canonical height on a given elliptic curve $E/ \Q$ over all number fields of a fixed degree. We then describe data collected using this method, and investigate related conjectures of Lehmer and Lang using these data.
\end{abstract}

\maketitle

\setcounter{page}{1}
\section{Introduction}
Let $E$ be an elliptic curve over a number field $K$. We denote by $\overline{K}$ a fixed algebraic closure of $K$ and by $\hat{h}$ the canonical height function on $E(\overline{K})$.
Recall that $\hat{h}(P) = 0$ if and only if $P$ is a torsion point. There is much interest in studying the canonical heights of non-torsion points. In particular, we have the following conjecture, which is known as Lehmer's conjecture because of its analogy with a conjecture of D.H. Lehmer from 1933 \cite{Leh}. It describes how the smallest possible height of a non-torsion point $P \in E(\overline{K})$ varies with the (minimal) field $K(P)$ over which $P$ is defined.

\begin{conjecture}[Lehmer]\label{conj:Lehmer}
    Let 
    \[C_E \coloneqq \inf
    \left\{ \hat{h}(P) \cdot [K(P) : K]\right\},\] 
    where the infimum ranges over the non-torsion points $P \in E(\overline{K}) - E(\overline{K})_{\textnormal{tors}}$. Then the constant $C_E$ satisfies $C_E>0$.
\end{conjecture}

The other primary conjecture describes how the smallest possible height of a non-torsion point $P \in E(\overline{K})$ defined over an extension of a given degree varies with the curve $E$. Denote by $j_E,\Delta_E$ the $j$-invariant and minimal discriminant of $E/K$. We write $N_{K/\Q} : K \rightarrow \Q$ for the norm map and see Definition \ref{def:heights} for the height function $h : \PP^1(\overline{K}) \rightarrow \R_{\geq 0}$. Consider the quantity $M_E = \max\{h(j_E),\log |N_{K/\Q} \Delta_E|,1\}$.

\begin{conjecture}[Lang]\label{conj:Lang}
    Let
    \[C_{K,d}\coloneqq \inf \left\{ \frac{\hat{h}(P)}{M_{E'}} \right\}, \]
    where the infimum ranges over all elliptic curves $E'/K$ and the non-torsion points $P \in E'(\overline{K}) - E'(\overline{K})_{\textnormal{tors}}$ for which $K(P)$ is contained in a degree $d$ extension of $K$. Then the constant $C_{K,d}$ satisfies $C_{K,d}>0$.
\end{conjecture}

Although there is theoretical progress on these conjectures
and their generalisations to abelian varieties over number fields, very little experimental work has been done investigating the values of $C_E$ and $C_{K,d}$. 
In this short paper, we describe a database of quadratic points of small height on 17,834 elliptic curves over the rationals $K=\Q$. In 728 of the cases, the point in the database is \emph{provably} the point of smallest height on the given elliptic curve over \emph{any} quadratic field. The computations to collect our data required just over 800 hours of CPU time. We use these data to investigate the constants in Conjectures \ref{conj:Lehmer} and \ref{conj:Lang}.

We proceed first with a brief background on heights, followed by a description of the theoretical results underlying the algorithm used to build our database. We then discuss some preliminary observations about the resulting data, and possible future work.


\section{Computing minimal heights over field extensions}\label{sec:TheoreticResults}
Let $K$ be a number field with fixed algebraic closure $\overline{K}$, and let $E$ be an elliptic curve over $K$, given by an affine Weierstrass equation with coefficients in $K$.

\begin{definition}\label{def:heights}
    Let $x : E(\overline{K}) \rightarrow \PP^1(\overline{K})$ denote the map taking the $x$-coordinate and $h : \PP^1(\overline{K}) \rightarrow \R_{\geq 0}$ the \emph{(absolute logarithmic) Weil height} on $\PP^1(\overline{K})$, as defined in \cite{Hindry-Silverman-Diophantine-Geometry}, Section B.2. By a standard abuse of notation, we also denote by $h : E(\overline{K}) \rightarrow \R_{\geq 0}$ the map defined by $P \mapsto h(x(P))$. \\
    We denote the \emph{canonical height} on $E/K$ by
    \begin{align*}
        \hat{h} : E(\overline{K}) &\rightarrow \R_{\geq 0}, \quad
        P \mapsto \lim_{n \rightarrow \infty} \frac{1}{4^n} h(2^nP).
    \end{align*}

\end{definition}

Recall that the canonical height is the unique quadratic form $E(\overline{K}) \rightarrow \R_{\geq 0}$ with the property that the function $P \mapsto |h(P) - \hat{h}(P)|$ is bounded.

Let $E$ be an elliptic curve over a number field $K$ and let $\mathscr{F}$ be a set of finite field extensions of $K$ with the following properties: 
\begin{itemize}
    \item If $F \in \mathscr{F}$ and $F' \subset F$, then $F' \in \mathscr{F}$.
    \item The set of degrees $\{[F:K] : F \in \mathscr{F}\}$ is finite.
\end{itemize}
Consider the infimum
\[C_{E,\mathscr{F}} := \inf_{F \in \mathscr{F}, P \in E(F) - E(F)_{\mathrm{tors}}} \left\{ \hat{h}(P)\cdot [F : K]\right\}.\]
\begin{remark}
    The first property ensures that whenever $F \in \mathscr{F}$ and $P \in E(F)$, the set $\mathscr{F}$ also contains the minimal field of definition $K(P)$ of $P$. The second property ensures that the subset of number fields in $\mathscr{F}$ of discriminant bounded by a given value is finite. In turn, using (for example) Lemma \ref{lem:large disc means large height} below, this implies that a Northcott property holds for all fields in $\mathscr{F}$: There are finitely many points of bounded height on $E$ over fields in $\mathscr{F}$. Thus the minimum height of such points exists and it follows that $C_{E,\mathscr{F}}>0$. The fact that $C_{E,\mathscr{F}}>0$ also follows directly from Theorem \ref{thm:finiteset} and it is predicted by Conjecture \ref{conj:Lehmer} since $C_{E,\mathscr{F}} \geq C_E$.
\end{remark}

 In this section we explain how to explicitly compute $C_{E,\mathscr{F}}$ using a lower bound on the Weil height $h(P)$ of the $x$-coordinate, and an upper bound on the difference $\left|h(P) - \hat{h}(P)\right|$ with the canonical height.
 
We proceed in two steps: First we determine a finite set $\mathscr{F}' \subset \mathscr{F}$ such that $C_{E,\mathscr{F}'} = C_{E,\mathscr{F}}$.
Then we explain how to solve the finite problem of determining $C_{E,\mathscr{F}'}$. Computational challenges arise when $\mathscr{F}'$ is large; we discuss these in the next section, where we consider the case $K = \Q$ and $\mathscr{F} = \{F/\Q : [F:\Q] \leq 2\}$.

As noted under Definition \ref{def:heights}, we can fix $B_E \in \R_{>0}$ such that 
\begin{equation}\label{bound between Weil ht and canonical ht}
    \left|h(P) - \hat{h}(P)\right| \leq B_E
\end{equation}
for all $P \in \cup_{F \in \mathscr{F}} E(F)$, see for example \cite{silverman1990difference} for an explicit value of $B_E$. For now, any $B_E$ satisfying (\ref{bound between Weil ht and canonical ht}) will do, but for our explicit computations it is useful to have $B_E$ as small as possible. We will use a modified version of the bound given in \cite{cremona2006height}, which we describe in Section \ref{subsec:ModifiedCPSBound2}.

\begin{lemma}\label{lem:large disc means large height}
    Let $D \in \R_{\geq 0}$, $F \in \mathscr{F}$, and $d=[F:K]$. Let $\delta_K$ be the number of Archimedean places of $K$. 
    Define $\Delta(D,E,F) \in \R_{>0}$ by 
    \[\Delta(D,E,F) := \exp\left( d\delta_K \log d + d(2d-2)B_E + (2d-2)D \right).\]
    If the discriminant $\Delta_F$ of $F$ satisfies
    $|\Delta_F| \geq \Delta(D,E,F)$,
    then
    $\hat{h}(P) \geq \frac{D}{d}$
    for all $P \in E(F)-E(F)_{\textnormal{tors}}$ satisfying $K(P) = F$. Further, if $[F:K] = [F' : K]$, then $\Delta(D, E, F) = \Delta(D, E, F')$.
\end{lemma}
\begin{proof}
    By Theorem 2 in \cite{silverman1984lower} we have 
    $h(P) \geq \frac{1}{2d-2}\left( \frac{1}{d}\log |\Delta_F| - \delta_K \log d \right).$
    The first part of the lemma follows by combining $|\Delta_F| \geq \Delta(D,E,F)$ and Equation (\ref{bound between Weil ht and canonical ht}). The second part of the lemma is clear from the definition of $\Delta(D,E,F)$.
\end{proof}

We can now reduce $\mathscr{F}$ to a finite set. 

\begin{theorem} \label{thm:finiteset}
    Let $D' \in \R_{\geq 0}$ be such that $C_{E,\mathscr{F}} \leq D'$ and
\[\mathscr{F'} = \left\{ F \in \mathscr{F} : |\Delta_{F}| \leq \Delta(D' ,E,F) \right\}.\]
    Then, $\mathscr{F'}$ is finite and $C_{E,\mathscr{F}'} = C_{E,\mathscr{F}}$.
\end{theorem}

\begin{proof}
    By our initial assumptions on $\mathscr{F}$, the set $\{[F:K] : F \in \mathscr{F}\}$ is finite. Therefore, it follows from the second part of Lemma \ref{lem:large disc means large height} that the maximum  $\Delta = \max\{\Delta(D',E,F) : F \in \mathscr{F}\}$ exists. The set $\mathscr{F}_\Delta = \left\{ F \in \mathscr{F} : |\Delta_{F}| \leq \Delta \right\}$ is finite by the Hermite–Minkowski Theorem, and hence its subset $\mathscr{F'} \subset \mathscr{F}_\Delta$ is also finite. The first part of Lemma \ref{lem:large disc means large height} implies that $C_{E,\mathscr{F}'} = C_{E,\mathscr{F}}$.
\end{proof}

In principle, we can therefore compute $C_{E,\mathscr{F}}$ as follows: Do an initial search to find  $F'\in \mathscr{F},P' \in E(F')$ with $K(P')=F'$ such that 
$D'=\hat{h}(P')[F':K]$ is \emph{small}. Then $C_{E,\mathscr{F}} \leq D'$ and we write $\mathscr{F}' \subset \mathscr{F}$ for the associated finite set of fields from Theorem \ref{thm:finiteset}. In theory any $F',P'$ work, but in practice it is worth spending more time in the initial search, as a smaller $D'$ decreases the number of fields in $\mathscr{F}'$ to be considered later. For each $F \in \mathscr{F}'$ do a finite search to find the points $P \in E(F)$ such that 
\begin{equation}\label{quick search equation}
    h(P) \leq \frac{D'}{[F:K]} + B_E.
\end{equation}
If $\mathscr{F}'$ is \emph{not too large}\footnote{What this means exactly depends on the efficiency of the used algorithms and the available memory and computing power. See also Section 4.} and we can list it explicitly, we obtain in this way the finite list of $F,P$ satisfying (\ref{quick search equation}), among which is a number field $F_E \in \mathscr{F}$ and a $P_E \in E(F_E)-E(F_E)_{\textnormal{tors}}$ such that $C_{E,\mathscr{F}} = \hat{h}(P_E) \cdot [F_E:K]$.

\subsection{A modified CPS height bound}\label{subsec:ModifiedCPSBound2}

Let $E$ be an elliptic curve over a number field $K$ and $\mathscr{F}$ a set of extensions of $K$ such that the set of degrees $\{[F:K] : F \in \mathscr{F}\}$ is finite. In this subsection we compute a bound $B_E \in \R_{\geq 0}$ such that $h(P) - \hat{h}(P) \leq B_E$ for all $P \in E(F)$ as $F$ ranges over $\mathscr{F}$. As mentioned, there is previous work (for example Silverman \cite{silverman1990difference} and Bruin \cite{Bruin13}) computes a bound on $h(P) - \hat{h}(P)$ for all $P \in E(\overline{K})$, but in practice a smaller bound is desirable. For a given $F \in \mathscr{F}$, Cremona, Prickett and Siksek describe a bound $B_{E,F}$ for $h(P)-\hat{h}(P)$ for all $P \in E(F)$ in Theorem 1 of \cite{cremona2006height}, which is small enough for our purposes, and we now explain how to modify it to work for all $F \in \mathscr{F}$ at once.

Indeed, for $F \in \mathscr{F}$, the bound $B_{E,F}$ is of the form
\[B_{E,F} = \frac{1}{[F:K]} \sum_v M_v,\]
where the sum ranges over the set of archimedean places of $F$ and the set of primes of $F$ for which $E/F_v$ has bad reduction. The $M_v$ are certain local invariants associated to $E/F_v$. Since the set of degrees $\{[F:K] : F \in \mathscr{F}\}$ is finite, and since there are only finitely many extensions of $\R$ and $\Q_p$ of any given degree, the set $\{ B_{E,F} : F \in \mathscr{F} \}$ attains its maximum and we can set $B_E = \max\{ B_{E,F} : F \in \mathscr{F} \}$.

\begin{remark}
    We have implemented the above procedure for computing $B_E$ in the case $K = \Q$ and $\mathscr{F} = \{F/\Q : [F:\Q] \leq 2\}$.
\end{remark}

\section{Computational Results}

We implemented the strategy in Section \ref{sec:TheoreticResults} in the case of quadratic fields using \textsc{Magma} version 21.2-2\cite{MR1484478} and Sagemath version 10.6 \cite{sagemath}. In particular, for every elliptic curve in the Cremona database \cite{lmfdb} of conductor at most 3,000, we conducted an initial search to find points of small height. We then calculated a bound $\Delta=\Delta(D', E, F)$ as in Lemma \ref{lem:large disc means large height}, using the height bound in Section \ref{subsec:ModifiedCPSBound2} as our $B_E$, as well as a bound $B$ on the logarithmic height of the points that needed to be searched as in Equation \eqref{quick search equation}. For curves for which $\Delta < 10^5$ and $B < 50$, we searched all possible $x$-coordinates to obtain provably the smallest point, using the Sagemath implementation of \cite{DoyleKrumm}. In order to keep the computations feasible, for curves where either bound exceeded the numbers described above, we searched only over quadratic fields with $|\Delta_K| \leq 1,000.$  We believe this choice is sufficient as in the provable cases the point of smallest height was usually found lying in a field with small discriminant. The resulting datasets and the code used to produce them are available at \href{https://github.com/EliOrvis/LehmersConjectureForECs}{https://github.com/EliOrvis/LehmersConjectureForECs}. The datasets contain the following fields: 
\begin{itemize}
    \item the Cremona label for the curve; 
    \item the discriminant of the quadratic field over which the point of smallest height over all quadratic fields is defined;
    \item the coordinates of the point of smallest height over all quadratic fields;
    \item the height of this point.
\end{itemize}

\begin{remark} In view of the abundance of data in the LMFDB on generators of the Mordell-Weil group of the elliptic curves $E$ in our database and of their quadratic twists, it is not necessary to conduct an initial point search to compute the first bound $\Delta$ as these can be found in the LMFDB. Similarly, one could use the LMFDB precomputed rational points defined on $E$ itself or one of its quadratic twists to perform the search for points over quadratic fields. Implementing these changes could improve our algorithm. We thank an anonymous referee for this suggestion. \end{remark} 

\subsection{Description of data} \label{subsect:datadescription}

We ran an initial search on 17,834 elliptic curves, which required just over 800 hours of CPU time running on a server operating Red Hat Enterprise Linux 8.10. We were able to then verify that we found the point of smallest height over all quadratic fields for 86 of these curves. For 542 curves, the initial search failed to find a point, and so there is no point for these curves in our dataset. Among the remaining curves, the first
curves in our list (ordered by conductor) for which the discriminant bound obtained by the initial search was too big were the curves with Cremona label 11a1 and 11a2, of conductor 11.

Among all curves in our dataset, the smallest height we found was the point 
\begin{equation}\label{eq:SmallestHeightPt}
(27, -119, 1) \quad \text{on the elliptic curve} \quad y^2 + xy + y = x^3 + x^2 - 2990x + 71147,    
\end{equation}
 which has Cremona label 1470l1, and height $0.0099641079999\hdots$.

We note that the point in \eqref{eq:SmallestHeightPt} was also found by Elkies \cite{Elkies}, although his normalization of the height makes his value half of ours. At the same time, Taylor found points of much smaller height on elliptic curves defined over quadratic fields \cite{TaylorTables} in unpublished work. Our methods, however, differ from both of these previous computations, in that we search broadly over elliptic curves by conductor, whereas these prior computations were targeted searches in families of elliptic curves likely to contain points of small height.


We also make some observations about the quadratic fields over which points of smallest height are defined. In our dataset, the point of smallest height that we found was defined over $\Q$ for 2,199 of the elliptic curves. The next most common fields were the two cyclotomic quadratic fields: $\Q(\sqrt{-3})$ with 1,610 elliptic curves and $\Q(\sqrt{-4})$ with 1,191 elliptic curves. This remained consistent when restricting to curves where our point is provably the smallest over any quadratic field: in this case, the most common field was $\Q(\sqrt{-3})$ with 20 curves, followed by $\Q(\sqrt{-4})$ with 14, and then $\Q$ with 11. Finally, we note that among curves where we have provably found the point of smallest height over all quadratic fields, this point always agrees with the point found in our initial search. Thus, we suspect that for many of the remaining curves, the point in our dataset is in fact the smallest over all quadratic fields.

\subsection{Remarks on Conjectures \ref{conj:Lehmer} and \ref{conj:Lang}}

We conducted some preliminary investigations into Conjectures \ref{conj:Lehmer} and \ref{conj:Lang} using the data we collected. The resulting charts can be found at \href{https://github.com/EliOrvis/LehmersConjectureForECs}{https://github.com/EliOrvis/LehmersConjectureForECs}. Unfortunately, we did not find a discernible relationship between the smallest height point in our dataset and either the conductor or the discriminant of the elliptic curve.

\clearpage
\subsection*{Acknowledgements}
We would like to thank our mentors, Nicole Looper and Shiva Chidambaram, for their guidance throughout the project. We would also like to thank Joseph Silverman, for his mentorship, and Andrew Sutherland, for his computational insights. We thank John Voight for referring us to the article \cite{DoyleKrumm}. We would like to thank the organizers Serin Hong, Hang Xue, Alina Bucur, Renee Bell, Brandon Levin, Anthony Várilly-Alvarado, Isabel Vogt and David Zureick-Brown of the 2024 Arizona Winter School on Abelian Varieties. Finally, we would like to thank the National Science Foundation and the Clay Mathematics Institute, and the anonymous referees for their helpful feedback.

\bibliographystyle{plain}
\bibliography{refs_consistent}

\end{document}